\documentclass[paper=a4,headsepline,DIV13]{scrartcl}

\usepackage[utf8]{inputenc}
\usepackage[T1]{fontenc}
\usepackage{lmodern}
\usepackage{amsmath,amsfonts,amssymb,amsthm,mathtools}
\usepackage{graphicx,subfigure}
\usepackage{braket,dsfont,url,enumerate}
\usepackage{hyperref}

\usepackage[textsize=small,obeyFinal]{todonotes}
\presetkeys{todonotes}{color=red!70}{}

\renewenvironment{abstract}{\minisec{Abstract}}{\par\vspace{.1in}}
\newenvironment{keywords}{\minisec{Key Words}}{\par\vspace{.1in}}
\newenvironment{AMS}{\minisec{AMS subject classification}}{\par\vspace{.1in}}

\theoremstyle{plain}
\newtheorem{theorem}{Theorem}[section]
\newtheorem{corollary}[theorem]{Corollary}
\newtheorem{lemma}[theorem]{Lemma}
\newtheorem{prop}[theorem]{Proposition}

\theoremstyle{remark}
\newtheorem{remark}[theorem]{Remark}

\theoremstyle{definition}
\newtheorem{assumption}{Assumption}

\numberwithin{equation}{section}

\newcommand{\R}{\mathds{R}}

\newcommand{\na}{\nabla}
\newcommand{\pa}{\partial}
\newcommand{\eps}{\varepsilon}
\newcommand{\Om}{\Omega}
 
\newcommand{\IOm}{I\times \Om}

\newcommand{\norm}[1]{\lVert#1\rVert}
\newcommand{\abs}[1]{\lvert#1\rvert}
\newcommand{\Ppol}[1]{\mathcal{P}_{#1}}
\newcommand{\lh}{\abs{\ln{h}}}
\newcommand{\lk}{\ln{\frac{T}{k}}}

\newcommand{\Xkzh}{X^{0,1}_{k,h}}

\DeclareMathOperator{\sgn}{sgn}
\renewcommand{\phi}{\varphi}

\begin{document}

\title{Optimal Error Estimates for\\
Fully Discrete Galerkin Approximations of\\
Semilinear Parabolic Equations}

\author{
  Dominik Meidner\footnotemark[2] \and Boris Vexler\footnotemark[2]
}

\markright{Meidner, Vexler: Estimates for Galerkin Approximations of Semilinear Parabolic Equations}

\maketitle

\renewcommand{\thefootnote}{\fnsymbol{footnote}}

\footnotetext[2]{Technical University of Munich, Department of Mathematics, Chair of Optimal
Control, Garching / Germany (meidner@ma.tum.de, vexler@ma.tum.de)}

\renewcommand{\thefootnote}{\arabic{footnote}}

\begin{abstract}
  We consider a semilinear parabolic equation with a large class of nonlinearities without any
  growth conditions. We discretize the problem with a discontinuous Galerkin scheme dG($0$) in time
  (which is a variant of the implicit Euler scheme) and with conforming finite elements in space.
  The main contribution of this paper is the proof of the  uniform boundedness of the discrete
  solution. This allows us to obtain optimal error estimates with respect to various norms.
\end{abstract}




\section{Introduction}

In this paper, we consider the following semilinear parabolic equation.
\begin{equation}\label{eq:heatEquation}
  \begin{aligned}
    \pa_t u(t,x)-\Delta u(t,x) + d(t,x,u(t,x))&= f(t,x) & (t,x) &\in \IOm,\\
    u(t,x) &= 0    & (t,x) &\in I\times\pa\Omega, \\
    u(0,x) &= u_0(x)   & x &\in \Omega.
  \end{aligned}
\end{equation}
Here, $\Om \subset \R^N$,  $N\in\set{2,3}$ is a convex polygonal/polyhedral domain, $I=(0,T)$ is a
time interval and $f$ is the right-hand side fulfilling a certain regularity requirement to be
specified later.

For the nonlinearity $d(t,x,u)$, we essentially assume that the partial derivative $\partial_u
d(t,x,u)$ is bounded from below for all $(t,x)\in I \times \Om$ and all $u \in \R$,
see~\eqref{ass:c}. But we do not require any growth conditions for $d$, see the next section for
details. The class of possible nonlinearities includes monotone nonlinearities like $d(u) = u^5$,
$d(u) = e^u$ or $d(u) = u^3 \abs{u}$ as well as FitzHugh-Nagumo or Allen–Cahn type nonlinearities
like $d(u) = u^3 - \alpha u$ with some positive $\alpha \in \R$.

For this class of problems (under a suitable assumption on the right-hand side $f$ and the initial
data $u_0$), it is possible to show the existence of a unique bounded solution $u$. The goal of
the paper is to prove the uniform boundedness of the discrete approximation $u_{kh}$ to $u$. To this
end, we discretize the equation with the discontinuous Galerkin  dG($0$) method in time and with 
conforming finite elements in space. The dG($0$) time discretization is known to be a variant of the
implicit Euler scheme, see Section~\ref{sec:disc} for details. For this type of discretization we
prove that $u_{kh}$ is uniformly bounded, i.e.,
\[
\norm{u_{kh}}_{L^\infty(I \times \Omega)} \le C
\]
with a constant $C$ independent of the discretization parameters $k$ and $h$, see
Theorem~\ref{th:BoundUkh}. Based on this result we are able to prove best-approximation-type error
estimates with respect to various norms. We provide such results in particular for the $L^2(I\times
\Omega)$, $L^\infty(I;L^2(\Omega))$, and $L^\infty(I \times \Omega)$ norms, cf. the
Theorems~\ref{th:L2L2},~\ref{th:L8L2}, and~\ref{th:L8L8}, respectively.

Let us review the related results in the literature.  In~\cite{EstepLarsson:1993,Thomee:1986,
ThomeeV_2006}, error estimates for discretization of the semilinear parabolic equation are derived
under the assumption that $d$ and $\partial_u d$ are uniformly bounded.
In~\cite{ChrysafinosHou:2002, LasisSueli:2007} growth conditions on $d$ (resp.\ $\partial_ud$) are
assumed for derivation of semi-discrete error estimates. For further results in a different setting
we refer to~\cite{AkrivisMakridakis:2004}.  The most related result is provided
in~\cite{NeitzelVexler:2010}, where the uniform boundedness of $u_{kh}$ is shown under a slightly
stronger condition $\partial_u d \ge 0$ (cf.~\eqref{ass:c}) in the two-dimensional setting. The
technique from~\cite{NeitzelVexler:2010} does not extend  to the three-dimensional
situation, due to the inverse inequality used there. Our method here strongly relies on recent
discrete maximal parabolic regularity estimates~\cite{LeykekhmanVexler:2017}, cf.
also~\cite{KovacsLiLubich:2016} for related results, and extends best approximation estimates
from~\cite{LeykekhmanD_VexlerB_2015d} to the semilinear equation.

Our error estimates being of independent interest are important for treatment of optimal control
problems. Some recent papers in this context (see, e.g.,~\cite{CasasMateosRoesch:2017,
CasasKruseKunisch:2017}) are restricted to two-dimensional domains only due to the lack of
corresponding results in the three-dimensional setting. Thus, our estimates allow to extend the
results of these papers to convex polyhedral domains $\Omega \subset \R^3$.

The outline of the paper is as follows: In Section~\ref{sec:ContProb}, we state the precise
functional analytic setting of the problem under consideration and formulate assumptions on the
nonlinearity $d$ and the remaining problem data. Under these assumptions, we prove Hölder continuity
of the solution $u$ to~\eqref{eq:heatEquation}. The discrete analog of~\eqref{eq:heatEquation} is
formulated in Section~\ref{sec:disc}. To this end, we introduce a time discretization by the discontinuous Galerkin dG($0$) scheme, whereas the discretization in space is done by means of classical Lagrange finite
elements. In this setting, we prove the unique solvability of the discrete nonlinear problem. In the
following Section~\ref{sec:aux}, we consider a linear auxiliary equation and its discrete analog.
For the solution to this linear discrete problem, we provide maximal parabolic estimates in various
norms, which will be the basis for analysis in the remaining two sections. In
Section~\ref{sec:bound}, we derive the main result of this paper, namely the boundedness of the
solution $u_{kh}$ to the discrete analog of~\eqref{eq:heatEquation}. Based on this, we provide in
the final Section~\ref{sec:estimates} optimal error estimates for the error between $u$ and $u_{kh}$ with respect to the
$L^2(\IOm)$, $L^\infty(I;L^2(\Om))$, and $L^\infty(\IOm)$ norms.


\section{Continuous Problem}\label{sec:ContProb}

To state the precise setting for the problem under consideration, we introduce the following
notation: for $r\in[1,\infty]$ and $l\in\set{-1,0}$, we denote the domain in $W^{l,r}(\Om)$ of the
negative Laplacian with homogeneous Dirichlet boundary conditions by
\[
  \operatorname{Dom}_{l,r}(-\Delta)=\Set{u\in W^{l,r}(\Om) | -\Delta u\in W^{l,r}(\Om)}.
\]
Further, for $p\in[1,\infty]$ , we define the space for the initial data by real interpolation as
\begin{equation}\label{eq:Xpr}
  U_{p,r}(\Om)=(L^r(\Om),\operatorname{Dom}_{0,r}(-\Delta))_{1-\frac1p,p}
\end{equation}

The following set of assumptions holds throughout the article.
\begin{assumption}~\label{ass:fduz}
  \begin{itemize}
    \item Let $f\in L^p(I;L^r(\Omega))$ for some $p\in(1,\infty)$ and $r\in\bigl(\frac
      N2,\infty\bigr)$ satisfying $\frac1p+\frac N{2r}<1$.
    \item Let $u_0 \in U_{p_0,r_0}(\Om)$ for some $p_0\in(1,\infty)$ and $r_0\in\bigl(\frac
      N2,\infty\bigr)$ satisfying $\frac1{p_0}+\frac N{2r_0}<1$.
  \end{itemize}
  Further, for the nonlinearity $d = d(t, x, u)\colon I \times\Omega\times\R\to\R$, we assume the
  following properties:
  \begin{subequations}
    \begin{itemize}
      \item $d$ is measurable with respect to $(t, x) \in I \times\Omega $ for all $u \in \R$ and
        continuously differentiable  with respect to $u$ for almost all $(t, x) \in  I \times\Omega$.
      \item It holds $d(\cdot,\cdot,0)=0$.
      \item $\partial_ud$ is locally bounded, i.e., for each $M>0$ there
        is $C_M>0$ such that
        \begin{equation}\label{ass:b}
          \abs{\partial_ud(t,x,u)}\le C_M
        \end{equation}
        for  almost all $(t,x)\in I\times\Omega$ and all $u\in[-M,M]$.
      \item There is $\gamma\ge0$ such that $d$ fulfills the relaxed monotonicity condition
        \begin{equation}\label{ass:c}
          \partial_ud(t,x,u)\ge -\gamma
        \end{equation}
        for almost all $(t,x)\in I\times\Omega$ and all $u\in\R$.
    \end{itemize}
  \end{subequations}
\end{assumption}

\begin{remark}
  A typical setting fulfilling the assumption on $u_0$ would be $u_0\in H^2(\Om)\cap H^1_0(\Om)$.
  Then, $u_0\in U_{p_0,r_0}(\Om)$ and the relation $\frac1{p_0}+\frac N{2r_0}<1$ is valid for
  $r_0=2$ and any $p_0>\frac{4}{4-N}$.
\end{remark}

\begin{remark}
  Each of the assumptions on $f$ and $u_0$ can be replaced independently by the following
  assumptions, see the corresponding Remarks~\ref{rem:1} and~\ref{rem:2} below.
  \begin{itemize}
    \item Let $f\in L^q(I;W^{-1,s}(\Omega))$ for $q\in(1,\infty)$ and $s\in\bigl(N,\infty\bigr)$
      satisfying $\frac1q+\frac N{2s}<\frac12$.
    \item Let $u_0 \in \widetilde
      U_{q_0,s_0}(\Om)=(W^{-1,s_0}(\Om),\operatorname{Dom}_{-1,s_0}(-\Delta))_{1-\frac1{q_0},q_0}$
      for $q_0\in(1,\infty)$ and $s_0\in\bigl(N,\infty\bigr)$ satisfying $\frac1{q_0}+\frac
      N{2s_0}<\frac12$.
  \end{itemize}
  A typical setting fulfilling this assumption on $u_0$ would be $u_0\in W^{1,s_0}_0(\Om)$ with some
  $s_0>N$. Then, $u_0\in \widetilde U_{q_0,s_0}(\Om)$ and the relation $\frac1{q_0}+\frac
  N{2s_0}<\frac12$ is valid for any $q_0>\frac{2s_0}{s_0-N}$.
\end{remark}

To state the existence and boundedness of the solution to~\eqref{eq:heatEquation}, we need the
following lemma.

\begin{lemma}\label{lem:Xpr}
  Under the assumptions on $p_0$ and $r_0$ from Assumption~\ref{ass:fduz}, there is $\alpha>0$ such that
  \[
    U_{p_0,r_0}(\Om)\hookrightarrow C^\alpha(\Om)\hookrightarrow L^\infty(\Om).
  \]
\end{lemma}
\begin{proof}
  By Assumption~\ref{ass:fduz}, there are $\varepsilon,\alpha >0$ such that
  $1-\frac1{p_0}-\varepsilon>\frac N{2r_0}+\frac\alpha2$.  Using~\cite[Theorems 1.3.3 and
  1.15.2]{Triebel:1978} as well as~\cite[Theorem 2.10]{DisserRehbergElst:2015}, we get
  \[
    \begin{aligned}
      (L^{r_0}(\Om),\operatorname{Dom}_{0,r_0}(-\Delta))_{1-\frac1{p_0},p_0}&\hookrightarrow
      (L^{r_0}(\Om),\operatorname{Dom}_{0,r_0}(-\Delta))_{1-\frac1{p_0}-\varepsilon,1}\\&\hookrightarrow
      \operatorname{Dom}_{0,r_0}((-\Delta)^{1-\frac1{p_0}-\varepsilon})\hookrightarrow C^\alpha(\Om).
    \end{aligned}
  \]
  By the definition of $U_{p_0,r_0}(\Om)$ from~\eqref{eq:Xpr}, this states the assertion.
\end{proof}

\begin{remark}\label{rem:1}
  Using~\cite[Lemma 4.8]{DisserRehbergElst:2015}, a corresponding result also holds for $\widetilde
  U_{q_0,s_0}(\Om)$ with $\frac1{q_0}+\frac N{2s_0}<\frac12$.
\end{remark}

\begin{prop}\label{prop:ContSol}
  Under Assumption~\ref{ass:fduz}, problem~\eqref{eq:heatEquation} admits a unique solution
  $u\in L^\infty(\IOm)$ with a priori estimate
  \[
    \norm{u}_{L^\infty(\IOm)}\le
    C\bigl\{\norm{f}_{L^p(I;L^r(\Om))}+\norm{u_0}_{L^\infty(\Om)}\bigr\}.
  \]
\end{prop}
\begin{proof}
  Property \eqref{ass:c} of Assumption~\ref{ass:fduz} implies
  $d(\cdot,\cdot,u)u=(d(\cdot,\cdot,u)-d(\cdot,\cdot,0))u\ge-\gamma u^2$. Further,
  Lemma~\ref{lem:Xpr} ensured $u_0\in L^\infty(\Om)$. This and the remaining assumptions imply the
  assumptions on $d$ made in~\cite{Casas:1997}.  Hence,~\cite[Theorem~5.1]{Casas:1997}  proves the
  assertion.  A similar result under the assumption that $f\in L^{\hat p}(\IOm)$ for $\hat p>\frac
  N2+1$ and $u_0\in L^\infty(\Om)$ can be found in~\cite[Lemma~A.1]{RaymondZidani:1999}.
\end{proof}

The goal of the remaining part of this section is to prove the Hölder continuity of the solution
of~\eqref{eq:heatEquation}. Before doing so, we need to establish some results for the following
linear homogeneous and inhomogeneous problems

\begin{equation}\label{eq:linRhs}
  \begin{aligned}
    \pa_t v(t,x)-\Delta v(t,x) &= g(t,x) & (t,x) &\in \IOm,\\
    v(t,x) &= 0    & (t,x) &\in I\times\pa\Omega,\\
    v(0,x) &= 0    & x &\in \Omega
  \end{aligned}
\end{equation}

and

\begin{equation}\label{eq:linU0}
  \begin{aligned}
    \pa_t w(t,x)-\Delta w(t,x) &= 0 & (t,x) &\in \IOm,\\
    w(t,x) &= 0    & (t,x) &\in I\times\pa\Omega,\\
    w(0,x) &= u_0(x)    & x &\in \Omega.
  \end{aligned}
\end{equation}

\begin{prop}\label{prop:ContHoelderRhs}
  Let $g\in L^p(I;L^r(\Om))$ with $\frac1p+\frac N{2r}<1$. Then, there are $\beta,\kappa>0$ depending
  on $p$ and $r$ such that the solution $v$ of~\eqref{eq:linRhs} fulfills $v\in C^\beta(I;C^\kappa(\Om))$
  with
  \[
    \norm{v}_{C^\beta(I;C^\kappa(\Om))}\le C \norm{g}_{L^p(I;L^r(\Om))}.
  \]
  Additionally , provided that $g\in L^{\hat p}(I;L^2(\Om))$ for some $1<\hat p<\infty$, it holds that
  $v\in W^{1,\hat p}(I;L^2(\Om))\cap L^{\hat p}(I;H^2(\Om))$ with the estimate
  \[
    \norm{\pa_t v}_{L^{\hat p}(I;L^2(\Om))}+\norm{\nabla^2 v}_{L^{\hat p}(I;L^2(\Om))}
    \le C_{\hat p}\norm{g}_{L^{\hat p}(I;L^2(\Om))}
  \]
  where $C_{\hat p}\le C\frac{\hat p^2}{\hat p-1}$.
\end{prop}
\begin{proof}
  The first result is proven, e.g., in~\cite[Theorem 3.1]{DisserRehbergElst:2015} setting $u_0=0$
  there. The second result can be found in~\cite[Lemma~2.1]{LeykekhmanVexler:2016}, which itself
  mainly relies on~\cite{AshyralyevSobolevski:1994} and~\cite{ElschnerRehbergSchmidt:2007}.
\end{proof}

\begin{prop}\label{prop:ContHoelderU0}
  Let $u_0\in U_{p_0,r_0}(\Om)$ with $\frac1{p_0}+\frac N{2r_0}<1$. Then, there are $\beta,\kappa>0$
  depending on $p_0$ and $r_0$ such that the solution $w$ of~\eqref{eq:linU0} fulfills $w\in
  C^{\beta}(I;C^{\kappa}(\Om))$ with
  \[
    \norm{w}_{C^{\beta}(I;C^{\kappa}(\Om))}\le C \norm{u_0}_{U_{p_0,r_0}(\Om)}.
  \]
  Additionally, provided that $u_0\in H^2(\Om)\cap H^1_0(\Om)$, it holds that $w\in
  W^{1,\infty}(I;L^2(\Om))\cap L^\infty(I;H^2(\Om))$ with the estimate
  \[
    \norm{\pa_t w}_{L^\infty(I;L^2(\Om))}+\norm{\nabla^2 w}_{L^\infty(I;L^2(\Om))}
    \le C\norm{\nabla^2 u_0}_{L^2(\Om)}.
  \]
\end{prop}
\begin{proof}
  The first result is proven, e.g., in~\cite[Theorem 3.1]{DisserRehbergElst:2015} setting $f=0$
  there. The second result follows from standard estimates for $z=\Delta w$ solving
  \[
  \begin{aligned}
  \partial_t z-\Delta z &=0&\quad&\text{in }\IOm,\\
  z(0)&=\Delta u_0&&\text{on }\Om
  \end{aligned}
  \]
  and elliptic regularity.
\end{proof}

\begin{remark}\label{rem:2}
  Using~\cite[Theorem 4.5]{DisserRehbergElst:2015}, the results of the
  Propositions~\ref{prop:ContHoelderRhs} and~\ref{prop:ContHoelderU0} can also be proven under the
  assumptions $f\in L^q(I;W^{-1,s}(\Om))$ with $\frac1q+\frac N{2s}<\frac12$ and $u_0\in\tilde
  U_{q_0,s_0}(\Om)$ with $\frac1{q_0}+\frac N{2{s_0}}<\frac12$.
\end{remark}

Based on these lemmas, we can derive the main result of this section, namely the Hölder continuity
of the solution of~\eqref{eq:heatEquation}.

\begin{theorem}\label{th:regU}
  Let Assumption~\ref{ass:fduz} be fulfilled. Then, there are $\beta,\kappa>0$ such that the
  solution $u$ of~\eqref{eq:heatEquation} fulfills $u\in C^\beta(I;C^\kappa(\Om))$ with a priori
  estimate
  \[
    \norm{u}_{C^\beta(I;C^\kappa(\Om))}\le C
    \bigl\{\norm{f}_{L^p(I;L^r(\Om))}+\norm{u_0}_{U_{p_0,r_0}(\Om)}\bigr\}.
  \]
\end{theorem}

\begin{proof}
  We write the solution $u$ of~\eqref{eq:heatEquation} as $u=v+w$ where $v$
  solves~\eqref{eq:linRhs} with right-hand side $g=f-d(\cdot,\cdot,u)$ and $w$
  solves~\eqref{eq:linU0}. Using Assumption~\ref{ass:fduz} and the boundedness of $u$ given by
  Proposition~\ref{prop:ContSol}, we get by~\eqref{ass:b}
  \[
    \begin{aligned}
      \norm{d(\cdot,\cdot,u)}_{L^p(I;L^r(\Om))}
      &=\norm{d(\cdot,\cdot,u)-d(\cdot,\cdot,0)}_{L^p(I;L^r(\Om))}\\
      &\le C\norm{u}_{L^\infty(\IOm)}\le
      C\bigl\{\norm{f}_{L^p(I;L^r(\Om))}+\norm{u_0}_{L^\infty(\Om)}\bigr\}.
    \end{aligned}
  \]
  Hence, $g$ lies in $L^p(I;L^r(\Om))$ and Proposition~\ref{prop:ContHoelderRhs} implies the
  existence of $\beta_1,\kappa_1>0$ such that
  \[
    \norm{v}_{C^{\beta_1}(I;C^{\kappa_1}(\Om))}\le C\norm{g}_{L^p(I;L^r(\Om))}\le C
    \bigl\{\norm{f}_{L^p(I;L^r(\Om))}+\norm{u_0}_{L^\infty(\Om)}\bigr\}.
  \]
  Further, by Proposition~\ref{prop:ContHoelderU0}, there are $\beta_2,\kappa_2>0$ such that
  \[
    \norm{w}_{C^{\beta_2}(I;C^{\kappa_2}(\Om))}\le C \norm{u_0}_{U_{p_0,r_0}(\Om)}.
  \]
  Then, setting $\beta=\min\set{\beta_1,\beta_2}$ and $\kappa=\min\set{\kappa_1,\kappa_2}$ and using
  Lemma~\ref{lem:Xpr} yields the assertion for $u=v+w$.
\end{proof}

\section{Discrete Problem}\label{sec:disc}

To introduce the time discontinuous Galerkin discretization for the problem, we partition  the
interval $(0,T]$ into subintervals $I_m = (t_{m-1}, t_m]$ of length $k_m = t_m-t_{m-1}$, where $0 =
t_0 < t_1 <\cdots < t_{M-1} < t_M =T$. The maximal and minimal time steps are denoted by $k
=\max_{m} k_m$ and $k_{\min}=\min_{m} k_m$, respectively.  

\begin{assumption}\label{ass:2}
  We impose the following conditions on the temporal mesh (as, e.g., in~\cite{LeykekhmanVexler:2017}
  or~\cite{DMeidner_RRannacher_BVexler_2011a}):
  \begin{itemize}
    \item There are constants $c_1,c_2>0$ independent of $k$ such that $k_{\min}\ge c_1k^{c_2}$.
    \item There is a constant $c>0$ independent of $k$ such that for all $m=1,2,\dots,M-1$ it
      holds $c^{-1}\le\frac{k_m}{k_{m+1}}\le c$.
    \item It holds $k\le\frac{1}{4}T$.
  \end{itemize}
  Further, let $\gamma\ge0$ be such that~\eqref{ass:c} holds. If $\gamma>0$, we make the following
  assumption on the smallness of $k$:
  \begin{itemize}
    \item There is $0<\rho<1$ such that $k$ fulfills $k\le\frac\rho\gamma$.
  \end{itemize}
  If $\gamma=0$, no further assumption on $k$ has to be made.
\end{assumption}

For the discretization in space with discretization parameter $h > 0$, let $\mathcal{T}$ denote  a
quasi-uniform triangulation of $\Om$  with mesh size $h$, i.e., $\mathcal{T} = \{\tau\}$ is a
partition of $\Om$ into cells (triangles or tetrahedrons) $\tau$ of diameter $h_\tau$ such that for
$h=\max_{\tau} h_\tau$,
\[
  \operatorname{diam}(\tau)\le h \le C |\tau|^{\frac{1}{N}}, \quad \forall \tau\in \mathcal{T}.
\]
Let $V_h$ be the set of all functions in $H^1_0(\Om)$ that are Lagrange polynomials of order $\nu\ge1$ on each
$\tau$. We consider the space-time finite element space
\[
  \Xkzh=\Set{v_{kh} \in L^2(I;V_h)| v_{kh,m}:=v_{kh}|_{I_m}\in \Ppol{0}(I_m;V_h), \ m=1,2,\dots,M},
\]
where $\Ppol{0}(I;V)$ is the space of constant polynomial functions in time with values in a Banach
space $V$.

Throughout, we denote by $P_h\colon L^2(\Om)\to V_h$ the spatial orthogonal $L^2$ projection and by
$R_h\colon H^1_0(\Om)\to V_h$ the spatial Ritz projection. Moreover, we introduce the
discrete Laplace operator $\Delta_h\colon V_h \to V_h$ defined by
\[
   (-\Delta_h v_h, \phi_h)_\Om = (\nabla v_h, \nabla \phi_h)_\Om \quad\forall\phi_h \in V_h.
\]
Further, we denote by $P_k$ the temporal
$L^2$ projection given for a function $v\in L^1(I)$ by
\[
  (P_kv)\bigr\rvert_{I_m}=\frac1{k_m}\int_{I_m}v(t)\,dt, \qquad m=1,2,\dots,M.
\]
Finally, the projection $\Pi_k$ is given for $v\in C(\bar I)$ by
\[
  (\Pi_kv)\bigr\rvert_{I_m}=v(t_m),\qquad m=1,2,\dots,M.
\]
The extension of these operators to space- and time-dependent functions is obvious. 

We will employ the following notation for time-dependent functions $v$:
\[
  v^+_m=\lim_{\eps\to 0^+}v(t_m+\eps), \quad v^-_m=\lim_{\eps\to 0^+}v(t_m-\eps), \quad [v]_m=v^+_m-v^-_m.
\]
Note, that by definition, for $v_{kh}\in\Xkzh$, it holds
\[
  v_{kh,m}^+=v_{kh,m+1}, \quad v^-_{kh,m}=v_{kh,m}, \quad [v_{kh}]_m=v_{kh,m+1}-v_{kh,m}.
\]

Based on these preparations, we define the bilinear form $B$ by
\begin{equation}\label{eq:B}
  B(u,\varphi)=\sum_{m=1}^M \langle \pa_t u,\varphi \rangle_{I_m \times \Omega} + (\na u,\na
  \varphi)_{\IOm}+\sum_{m=2}^M([u]_{m-1},\varphi_{m-1}^+)_\Om+(u_{0}^+,\varphi_{0}^+)_\Om,
\end{equation}
where $( \cdot,\cdot )_{\Omega}$ and $( \cdot,\cdot )_{I_m \times \Omega}$ are the usual $L^2$ space
and space-time inner products, $\langle \cdot,\cdot \rangle_{I_m \times \Omega}$ is the duality
pairing between $ L^2(I_m;H^{-1}(\Omega))$ and $ L^2(I_m;H^{1}_0(\Omega))$. Rearranging the terms in
\eqref{eq:B}, we obtain an equivalent (dual) expression for $B$:
\begin{equation}\label{eq:B_d}
  B(u,\varphi)= - \sum_{m=1}^M \langle u,\pa_t \varphi \rangle_{I_m \times \Omega} + (\na u,\na
  \varphi)_{\IOm}-\sum_{m=1}^{M-1} (u_m^-,[\varphi]_m)_\Om + (u_M^-,\varphi_M^-)_\Om.
\end{equation}
We note, that the first sum in~\eqref{eq:B} vanishes for $u=u_{kh}\in\Xkzh$ and the first sum
in~\eqref{eq:B_d} for$\phi=\phi_{kh}\in\Xkzh$, respectively. Hence, on $\Xkzh\times\Xkzh$, the
semilinear form $B$ can be reduced to
\begin{equation}\label{eq:B_primal}
  B(u_{kh},\varphi_{kh})=(\nabla
  u_{kh},\nabla\varphi_{kh})_{\IOm}+\sum_{m=2}^M([u_{kh}]_{m-1},\varphi_{kh,m})_\Om+(u_{kh,1},\varphi_{kh,1})_\Om
\end{equation}
and
\begin{equation}\label{eq:B_dual}
  B(u_{kh},\varphi_{kh})=(\nabla
  u_{kh},\nabla\varphi_{kh})_{\IOm}-\sum_{m=1}^{M-1}(u_{kh,m},[\varphi_{kh}]_m)_\Om+(u_{kh,M},\varphi_{kh,M})_\Om.
\end{equation}

Then, we define the fully discrete cG($1$)dG($0$) approximation $u_{kh} \in \Xkzh$
of~\eqref{eq:heatEquation} by
\begin{equation}\label{eq:discHeat}
  B(u_{kh},\varphi_{kh}) +
  (d(\cdot,\cdot,u_{kh}),\varphi_{kh})_{\IOm}=(f,\varphi_{kh})_{\IOm}+(u_0,\varphi_{kh,1})_\Om \quad\forall \varphi_{kh}\in \Xkzh.
\end{equation}


\begin{theorem}\label{th:exist_ukh}
  Under the Assumptions~\ref{ass:fduz} and~\ref{ass:2}, there is a unique solution $u_{kh}\in\Xkzh$
  of~\eqref{eq:discHeat}.
\end{theorem}

\begin{proof}
  Using~\eqref{eq:B_primal}, problem~\eqref{eq:discHeat} can be written as time stepping scheme for
  $u_{kh,m}=u_{kh}\bigr\rvert_{I_m}$ for $m=1,2,\dots, M$ as follows:
  \[
      k_m(\nabla u_{kh,m},\nabla
      \varphi_h)_\Om+(u_{kh,m}+k_m\bar d_m(\cdot,u_{kh,m}),\varphi_h)_\Om=(u_{kh,m-1}+k_m\bar f_m,\varphi_h)_\Om\quad\forall \varphi_h\in V_h,
  \]
  where $u_{kh,0}=P_hu_0$ and the mean values $\bar d_m$ and $\bar f_m$ are given on $I\times\Om$ by
  \[
    \bar d_m(x,u)=\frac1{k_m}\int_{I_m}d(t,x,u)\,dt~\text{for}~u\in\R\quad\text{and}\quad\bar
    f_m(x)=\frac1{k_m}\int_{I_m}f(t,x)\,dt.
  \]
  Hence, in each time step, the following discrete semilinear elliptic equation for $u_{kh,m}$ with
  given $u_{kh,m-1}$ has to be solved:
  \begin{equation}\label{eq:3}
      k_m(\nabla u_{kh,m},\nabla \varphi_h)_\Om+(\tilde
      d_m(\cdot,u_{kh,m}),\varphi_h)_\Om=(u_{kh,m-1}+k_m\bar f_m,\varphi_h)_\Om\quad\forall
      \varphi_h\in V_h.
  \end{equation}
  The nonlinearity $\tilde d_m$ is given for $u\in\R$ as $\tilde d_m(\cdot,u)=u+k_m\bar
  d_m(\cdot,u)$. Hence, Assumption~\ref{ass:2} and~\eqref{ass:c} imply $\partial_u\tilde
  d(\cdot,u)\ge1-k_m\gamma\ge1-\rho>0$ for $\gamma>0$ and $\partial_u\tilde d(\cdot,u)\ge1$
  independent of $k_m$ for $\gamma=0$.  The remaining assumptions on $d$ carry over to $\tilde d$
  and ensures the unique solvability of~\eqref{eq:3} for $m=1,2,\dots,M$ by application of Brouwer's
  fixed-point theorem, see, e.g.,~\cite{CasasMateos:2002}.
\end{proof}

\section{Discrete maximal parabolic estimates for a linear auxiliary equation}\label{sec:aux}

For given $g\in L^1(\IOm)$, we consider the discrete linear auxiliary equation for $v_{kh}\in\Xkzh$
\begin{equation}\label{eq:vkh}
  B(v_{kh},\varphi_{kh}) + (b
  v_{kh},\varphi_{kh})_{\IOm}=(g,\varphi_{kh})_{\IOm} \quad \forall
  \varphi_{kh}\in \Xkzh
\end{equation}
with a coefficient $b \in L^\infty(\IOm)$ fulfilling $b(t,x) \ge -\gamma$ for $\gamma\ge 0$ from
Assumption~\ref{ass:fduz} and almost all $(t,x)\in I\times\Omega$.

For the solution $v_{kh}$ of~\eqref{eq:vkh}, discrete maximal parabolic estimates in various norms
are available in the literature in the case $b=0$, see~\cite{LeykekhmanVexler:2017}. In this
section, we extend these results to the case $b\neq0$.  The  extended results will be used later in
the Section~\ref{sec:bound} and~\ref{sec:estimates} to prove the results for the semilinear problem. 

Before doing so, we start with an existence result for~\eqref{eq:vkh}.

\begin{theorem}
  Under Assumption~\ref{ass:2}, there is a unique solution $v_{kh}\in\Xkzh$
  of~\eqref{eq:vkh}.
\end{theorem}
\begin{proof}
  By setting $d(\cdot,\cdot,v_{kh})=bv_{kh}$, the asseretion follwos directly from
  Theorem~\ref{th:exist_ukh}.
\end{proof}

\begin{lemma}\label{lemma:1}
  Let Assumption~\ref{ass:2} be fulfilled and $g \in L^1(I;L^2(\Omega))$. Then, for the solution
  $v_{kh}\in\Xkzh$ of~\eqref{eq:vkh} there holds
  \[
    \norm{v_{kh}}_{L^\infty(I;L^2(\Omega))} \le C \norm{g}_{L^1(I;L^2(\Omega))}
  \]
  with a constant $C$ independent of $h$, $k$, $g$, and $b$.
\end{lemma}
\begin{proof}~
  We consider the dual problem for $z_{kh} \in \Xkzh$ given by
  \[
    B(\varphi_{kh},z_{kh})  + (b \varphi_{kh},z_{kh})_{\IOm} =
    (v_{kh,M},\varphi_{kh,M})_\Om\quad\forall\varphi_{kh}\in\Xkzh.
  \]
  Using~\eqref{eq:B_dual}, $z_{kh,m}$ satisfies for $m=M-1,M-2,\dots,1$ the scheme
  \begin{equation}\label{eq:4}
    k_m(\nabla \phi_h,\nabla
    z_{kh,m})_\Om+(z_{kh,m}+k_m\bar b_mz_{kh,m},\varphi_h)_\Om=(z_{kh,m+1},\varphi_h)_\Om\quad\forall \varphi_h\in V_h,
  \end{equation}
  where $z_{kh,M}=v_{kh,M}$ and $\bar b_m$ is given as before by
  \[
    \bar b_m(x)=\frac1{k_m}\int_{I_m}b(t,x)\,dt.
  \]

  To proceed, we will first prove the boundedness of $z_{kh}$ in $L^\infty(I;L^2(\Om))$. To this
  end, we employ the discrete transformation argument from~\cite{LeykekhmanVexler:2017:II}.  For
  $\mu>0$ a sufficient large number to be chosen later let $y_{kh,m}$ be defined as
  \[
    y_{kh,m}=z_{kh,m}\prod_{l=m}^M\frac1{1+\mu k_l},\qquad m=1,2,\dots,M.
  \]
  Then, by~\eqref{eq:4}, we get
  \begin{multline*}
    k_m\prod_{l=m}^M(1+\mu k_l)(\nabla \phi_h,\nabla
    y_{kh,m})_\Om+\prod_{l=m}^M(1+\mu k_l)(y_{kh,m}+k_m\bar
    b_my_{kh,m},\varphi_h)_\Om\\=\prod_{l=m+1}^M(1+\mu k_l)(y_{kh,m+1},\varphi_h)_\Om \quad\forall \varphi_h\in V_h.
  \end{multline*}
  Dividing both sides by $\prod_{l=m+1}^M(1+\mu k_l)$ yields
  \begin{multline*}
    k_m(1+\mu k_m)(\nabla \phi_h,\nabla
    y_{kh,m})_\Om+(1+\mu k_m)(y_{kh,m}+k_m\bar
    b_my_{kh,m},\varphi_h)_\Om\\=(y_{kh,m+1},\varphi_h)_\Om \quad\forall \varphi_h\in V_h,
  \end{multline*}
  which can be rewritten as
  \begin{equation}\label{eq:5}
    k_m(1+\mu k_m)(\nabla \phi_h,\nabla
    y_{kh,m})_\Om+(y_{kh,m}+k_m\tilde b_my_{kh,m},\varphi_h)_\Om=(y_{kh,m+1},\varphi_h)_\Om\quad\forall \varphi_h\in V_h
  \end{equation}
  with $\tilde b_m=\bar b_m+ \mu(1+k_m\bar b_m)$. Using Assumption~\ref{ass:2} and choosing
  $\mu\ge\frac{\gamma}{1-\rho}$ yields
  \[
    \tilde b_m\ge -\gamma+\mu(1-k_m\gamma)\ge-\gamma+\mu(1-\rho)\ge0.
  \]

  Then, by testing~\eqref{eq:5} with $\varphi_h=y_{hk,m}$, we get
  $\norm{y_{kh,m}}^2_{L^2(\Om)}\le(y_{kh,m+1},y_{kh,m})_\Om$, which implies
  $\norm{y_{kh,m}}_{L^2(\Om)}\le \norm{y_{kh,m+1}}_{L^2(\Om)}$.  Using this recursively for
  $m=1,2,\dots,M-1$, we get
  \[
    \norm{y_{kh,1}}_{L^2(\Om)} \le \norm{y_{kh,M}}_{L^2(\Om)}=\norm{v_{kh,M}}_{L^2(\Om)}.
  \]
  Transforming back to $z_{kh,m}$ and using $1+\mu k_l\le e^{\mu k_l}$ yields
  \[
    \norm{z_{kh,1}}_{L^2(\Om)}=\norm{y_{kh,1}}_{L^2(\Om)}\prod_{l=1}^M(1+\mu k_l)\le e^{\mu T}\norm{v_{kh,M}}_{L^2(\Om)}
  \]
  and hence
  \[
    \norm{z_{kh}}_{L^\infty(I;L^2(\Om))}\le e^{\mu T} \norm{v_{kh,M}}_{L^2(\Om)}.
  \]
  Using this and~\eqref{eq:vkh}, we obtain
  \[
    \begin{aligned}
      \norm{v_{kh,M}}_{L^2(\Omega)}^2 
      &= B(v_{kh},z_{kh}) +(b v_{kh},z_{kh})_{\IOm} = (g,z_{kh})_{\IOm} \\
      &\le \norm{g}_{L^1(I;L^2(\Omega))} \norm{z_{kh}}_{L^\infty(I;L^2(\Omega))} \le
      e^{\mu T} \norm{g}_{L^1(I;L^2(\Omega))}\norm{v_{kh,M}}_{L^2(\Omega)},
  \end{aligned}
  \]
  which completes the proof.
\end{proof}

The next lemma provides a discrete maximal parabolic estimate for $v_{kh}$ with respect to the
$L^\infty(I;L^2(\Om))$ norm.

\begin{lemma}\label{lemma:2}
  Let Assumption~\ref{ass:2} be fulfilled and $g \in L^\infty(I;L^2(\Omega))$. Then, for the
  solution $v_{kh}\in\Xkzh$ of~\eqref{eq:vkh} there holds
  \[
    \norm{\Delta_h v_{kh}}_{L^\infty(I;L^2(\Omega))}+\max_{1\le m\le
    M}\left\lVert\frac{[v_{kh}]_{m-1}}{k_m}\right\rVert_{L^2(\Omega)} \le C \lk \bigl\{1 +
    \norm{b}_{L^\infty(\IOm)}\bigr\} \norm{g}_{L^\infty(I;L^2(\Omega))}
  \]
  with a constant $C$ independent of $h$, $k$, $g$, and $b$.
\end{lemma}
\begin{proof}
  The solution $v_{kh}\in\Xkzh$ of~\eqref{eq:vkh} fulfills
  \[
    B(v_{kh},\varphi_{kh}) =(\tilde g,\varphi_{kh})_{\IOm}\quad \forall \varphi_{kh}\in \Xkzh
  \]
  with $\tilde g = g - b v_{kh}$.
  Using Lemma~\ref{lemma:1}, we can estimate
  \[
    \begin{aligned}
      \norm{\tilde g}_{L^\infty(I;L^2(\Omega))}
      &\le \norm{g}_{L^\infty(I;L^2(\Omega))} + \norm{b}_{L^\infty(\IOm)} \norm{v_{kh}}_{L^\infty(I;L^2(\Omega))}\\
      &\le \norm{g}_{L^\infty(I;L^2(\Omega))} + \norm{b}_{L^\infty(\IOm)}
      \norm{g}_{L^1(I;L^2(\Omega))}\\
      &\le C \bigl\{1 + \norm{b}_{L^\infty(\IOm)}\bigr\}\norm{g}_{L^\infty(I;L^2(\Omega))}.
    \end{aligned}
  \]
  Applying the discrete maximal parabolic regularity result of~\cite[Theorem 2 and Corollary
  2]{LeykekhmanVexler:2017}, we obtain the desired estimate for $v_{kh}$.
\end{proof}

Before continuing with estimates for the solution of~\eqref{eq:vkh}, we recall for completeness two
well-known results for finite element functions.

\begin{lemma}\label{lem:DeltaH}
  For any $w_h\in V_h$ it, holds 
  \[
    \norm{w_h}_{L^\infty(\Om)}\le C\norm{\Delta_hw_h}_{L^2(\Om)} \qquad\text{and}\qquad
    \norm{w_h}_{L^2(\Om)}\le C\norm{\Delta_hw_h}_{L^1(\Om)}.
  \]
\end{lemma}
\begin{proof}
  Let $w\in H^1_0(\Om)$ given as the solution of
  \[
    (\nabla w,\nabla \phi)_\Om=(-\Delta_hw_h,\phi)_\Om\quad\forall\phi\in H^1_0(\Om).
  \]
  Note, that by construction, it holds $R_hw=w_h$ for the Ritz projection $R_h$.  Elliptic
  regularity yields $w\in H^2(\Om)$ with $\norm{\nabla^2 w}_{L^2(\Om)}\le
  C\norm{\Delta_hw_h}_{L^2(\Om)}$. Further, it holds $\norm{w}_{L^2(\Om)}\le
  C\norm{\Delta_hw_h}_{L^1(\Om)}$ For the first assertion, let $i_h\colon C(\bar\Omega)\to V_h$ be
  the nodal interpolant. By standard estimates for $w_h-w$ and the interpolation error $w-i_hw$ as
  well as an inverse estimate, we get
  \[
    \begin{aligned}
      \norm{w_h}_{L^\infty(\Om)}
      &\le \norm{w_h-i_hw}_{L^\infty(\Om)}+\norm{i_hw-w}_{L^\infty(\Om)}+\norm{w}_{L^\infty(\Om)}\\
      &\le Ch^{-\frac N2}\bigl\{\norm{w_h-w}_{L^2(\Om)}+\norm{w-i_hw}_{L^2(\Om)}\bigr\} + \norm{i_hw-w}_{L^\infty(\Om)}+\norm{w}_{L^\infty(\Om)}\\
      &\le C(h^{2-\frac N2}+1)\norm{\nabla^2 w}_{L^2(\Om)}\le C\norm{\Delta_hw_h}_{L^2(\Om)}.
    \end{aligned}
  \]
  Similarly, we get for the second assertion that
  \[
    \begin{aligned}
      \norm{w_h}_{L^2(\Om)}
      &\le \norm{w_h-w}_{L^2(\Om)}+\norm{w}_{L^2(\Om)}\le
      C\bigl\{h^2\norm{\Delta_hw_h}_{L^2(\Om)}+\norm{\Delta_hw_h}_{L^1}\bigr\}\\
      &\le C(h^{2-\frac N2}+1)\norm{\Delta_hw_h}_{L^1(\Om)}\le C\norm{\Delta_hw_h}_{L^1(\Om)}.
    \end{aligned}
  \]
  This completes the proof.
\end{proof}

The next lemma provides a discrete maximal parabolic estimate for $v_{kh}$ with respect to the
$L^1(\IOm)$ norm.

\begin{lemma}\label{lemma:4}
  Let Assumption~\ref{ass:2} be fulfilled and $g \in L^1(\IOm)$. Then, for the solution
  $v_{kh}\in\Xkzh$ of~\eqref{eq:vkh} there holds
  \[
    \norm{\Delta_h  v_{kh}}_{L^1(\IOm)}+\sum_{m=1}^M \norm{[v_{kh}]_{m-1}}_{L^1(\Om)} \le C
    \left(\lk\right)^2 \bigl\{1+ \norm{b}_{L^\infty(\IOm)}^2\bigr\} \norm{g}_{L^1(\IOm)}
  \]
  with a constant $C$ independent of $h$, $k$, $g$, and $b$.
\end{lemma}

\begin{proof}
  We consider the dual problem for $z_{kh} \in \Xkzh$ given by
  \[
    B(\varphi_k,z_{kh})  + (b \varphi_k,z_{kh})_{\IOm} =
    (\varphi_{kh},\sgn v_{kh})_{\IOm}\quad\forall\varphi_{kh}\in\Xkzh.
  \]
  Then, it holds
  \[
    \norm{v_{kh}}_{L^1(\IOm)}=B(v_{kh},z_{kh})+(b v_{kh},z_{kh})_{\IOm}=(g,z_{kh})_{\IOm}\le
    \norm{g}_{L^1(\IOm)}\norm{z_{kh}}_{L^\infty(\IOm)}.
  \]
  By Lemma~\ref{lemma:2} applied to the dual solution $z_{kh}$ and Lemma~\ref{lem:DeltaH} applied
  separately to $w_h=z_{kh,m}$ for $m=1,2,\dots,M$, we get
  \[
    \begin{aligned}
      \norm{z_{kh}}_{L^\infty(\IOm)}&\le C \norm{\Delta_h  z_{kh}}_{L^\infty(I;L^2(\Omega))}\le C\ \lk
      \bigl\{1 + \norm{b}_{L^\infty(\IOm)}\bigr\}\norm{\:\!\sgn v_k}_{L^\infty(I;L^2(\Om))}\\
      &\le C \lk \bigl\{1 +\norm{b}_{L^\infty(\IOm)}\bigr\}
    \end{aligned}
  \]
  and consequently
  \begin{equation}\label{eq:2}
    \norm{v_{kh}}_{L^1(\IOm)}\le C \lk \bigl\{1 + \norm{b}_{L^\infty(\IOm)}\bigr\} \norm{g}_{L^1(\IOm)}.
  \end{equation}
  As before, this implies for $\tilde g = g - b v_{kh}$ that
  \[
    \norm{\tilde g}_{L^1(\IOm)}\le C \lk \bigl\{1 +
    \norm{b}_{L^\infty(\IOm)}^2\bigr\}\norm{g}_{L^1(\IOm)},
  \]
  which yields the assertion again by means of~\cite[Theorem 2 and Corollary
  2]{LeykekhmanVexler:2017}.
\end{proof}

\section{Boundedness of the Discrete Solution}\label{sec:bound}

In this section, we derive the boundedness of the solution $u_{kh}$ to~\eqref{eq:discHeat} in
$L^\infty(\IOm)$. In the case $N=2$, this was already proven in~\cite{NeitzelVexler:2010} using a
different approach than used here. The technique employed there does not extend to the
three-dimensional situation, due to the used inverse inequality.

First, we introduce a modified nonlinearity $d_R$ with bounded derivative $\partial_ud_R$, To this
end, let for $R>0$ the nonlinearity $d_R$ be defined by
\[
  d_R(t,x,u) =  \begin{cases}
    d(t,x,R)+(u-R)\partial_u d(t,x,R),&\text{for }u>R,\\
    d(t,x,u),&\text{for }\abs{u}\le R,\\
    d(t,x,-R)+(u+R)\partial_u d(t,x,-R),&\text{for }u<-R.
  \end{cases}
\]
Further, let $u^R$ and $u_{kh}^R$ be the solutions of the continuous problem~\eqref{eq:heatEquation}
and the discrete problem~\eqref{eq:discHeat} with $d_R$ instead of $d$.
Assumption~\eqref{ass:b} on the local boundedness of $\partial_ud$ implies the global boundedness of
\[
  \partial_u d_R(t,x,u) =  \begin{cases}
    \partial_u d(t,x,R),&\text{for }u> R,\\
    \partial_u d(t,x,u),&\text{for }\abs{u}\le R,\\
    \partial_u d(t,x,-R),&\text{for }u< -R
  \end{cases}
\]
by a constant $C_R$ depending on $R$:
\begin{equation}\label{eq:boundDR}
  \abs{\partial_ud_R(t,x,u)}\le C_R \quad\text{for almost all }(t,x)\in\IOm\text{ and all }u\in\R.
\end{equation}
Additionally, by~\eqref{ass:c}, it holds
\begin{equation}\label{eq:gammaDR}
  \partial_ud_R(t,x,u)\ge-\gamma.
\end{equation}

In the following lemma, we state an quasi best approximation result the error between $u^R$ and
$u_{kh}^R$ with respect to the $L^\infty(\IOm)$ norm:

\begin{lemma}\label{lem:estimate}
  Let the Assumption~\ref{ass:fduz} and~\ref{ass:2} be fulfilled, $u^R$ be the solution
  of~\eqref{eq:heatEquation}, and $u_{kh}^R\in\Xkzh$ be the solution of~\eqref{eq:discHeat} each
  with $d_R$ instead of $d$. Then, it holds
  \[
    \norm{u^R-u^R_{kh}}_{L^\infty(\IOm)}\le
    C_R\lh\left(\lk\right)^2\norm{u^R-\chi_{kh}}_{L^\infty(\IOm)}
  \]
  for any $\chi_{kh}\in\Xkzh$.
\end{lemma}
\begin{proof}
  Let $\chi_{kh}$ be an arbitrary but fixed element of $\Xkzh$.  We decompose the error $e =u^R -
  u_{kh}^R$ as
  \[
    e = (u^R - \chi_{kh}) + (\chi_{kh} - u_{kh}^R) = \eta + \xi_{kh},
  \]
  By Galerkin orthogonality, there holds
  \[
    B(e,\varphi_{kh}) + (d_R(\cdot,\cdot,u^R) - d_R(\cdot,\cdot,u_{kh}^R),\varphi_{kh})_{\IOm} = 0
    \qquad\forall \varphi_{kh}\in\Xkzh
  \]
  and therefore
  \begin{multline}\label{eq:Galerkin}
    B(\xi_{kh},\varphi_{kh}) + (d_R(\cdot,\cdot,\chi_{kh}) -
    d_R(\cdot,\cdot,u_{kh}^R),\varphi_{kh})_{\IOm}\\= -
    B(\eta,\varphi_{kh}) - (d_R(\cdot,\cdot,u^R) - d_R(\cdot,\cdot,\chi_{kh}),\varphi_{kh})_{\IOm}
  \end{multline}
  for all $\varphi_{kh}\in\Xkzh$.  To formulate an appropriate dual problem, we define the coefficient
  $b$ by
  \[
    b = \int_0^1 \partial_ud_R(\cdot,\cdot,u_{kh}^R + s (\chi_{kh} - u_{kh}^R))\, ds.
  \]
  By~\eqref{eq:boundDR}, it follows $\norm{b}_{L^\infty(I\times \Omega)} \le C_R$
  and~\eqref{eq:gammaDR} implies $b(t,x) \ge -\gamma$ for almost all $(t,x)\in I\times\Omega$.
  Further, by construction, it holds
  \[
    b\xi_{kh} = d_R(\cdot,\cdot,\chi_{kh}) - d_R(\cdot,\cdot,u_{kh}^R).
  \]

  We will estimate $\xi_{kh,M}(x_0)$ by using a duality argument. To this end, let
  $\tilde\delta_{x_0}\colon\Omega\to\R$ be a smoothed Dirac function with support contained in a
  single spatial cell $\bar\tau\ni x_0$ fulfilling
  \[
    \int_\tau \tilde\delta_{x_0}(x)\chi(x)\,dx=\chi(x_0)\quad\forall\chi\in
    \Ppol{1}(\tau)\quad\text{and}\quad\norm{\tilde\delta_{x_0}}_{L^1(\Om)}\le C.
  \]
  The explicit construction of such a function is given for instance in
  ~\cite[Appendix]{SchatzWahlbin:1995}.  Further, let $\theta_M\colon I\to\R$ be a smooth function
  with support contained in $I_M$ and fulfilling $\theta_M\ge 0$ as well as
  \[
    \int_{I_M}\theta_M(t)\,dt=1.
  \]
  Them, let $z_{kh} \in \Xkzh$ be given as solution of
  \[
    B(\varphi_{kh},z_{kh}) + (b \varphi_{kh},z_{kh})_{\IOm} =
    (\theta_M\tilde\delta_{x_0},\varphi_{kh})_{\IOm},\quad\forall\varphi_{kh}\in\Xkzh.
  \]
  Using~\eqref{eq:Galerkin}, we obtain
  \begin{equation}\label{eq:1}
    \begin{aligned}
      \xi_{kh,M}(x_0)&=(\theta_M\tilde\delta_{x_0},\xi_{kh})_{\IOm} =
      B(\xi_{kh},z_{kh})+(b\xi_{kh},z_{kh})_{\IOm}\\
      &=B(\xi_{kh},z_{kh})+(d_R(\cdot,\cdot,\chi_{kh}) - d_R(\cdot,\cdot,u_{kh}^R),z_{kh})_{\IOm}\\
      &=-B(\eta,z_{kh})-(d_R(\cdot,\cdot,u^R) - d_R(\cdot,\cdot,\chi_{kh}),z_{kh})_{\IOm}\\
      &=-(\nabla\eta,\nabla z_{kh})_{\IOm}+\sum_{m=1}^{M} (\eta_m,[z_{kh}]_m)_\Om-
      (d_R(\cdot,\cdot,u^R) - d_R(\cdot,\cdot,\chi_{kh}),z_{kh})_{\IOm},
    \end{aligned}
  \end{equation}
  where $\eta_m=u^R(t_m)-\chi_{kh,m}$.
  For the first term on the right-hand side of~\eqref{eq:1}, we get
  \[
    \begin{aligned}
      \abs{(\nabla\eta,\nabla z_{kh})_{\IOm}}\
      &= \abs{(\nabla R_h\eta,\nabla z_{kh})_{\IOm}}=\abs{(R_h\eta,\Delta_h z_{kh})_{\IOm}}\\
      &\le\norm{R_h\eta}_{L^\infty(\IOm)}\norm{\Delta_hz_{kh}}_{L^1(\IOm)}\\
      &\le C\lh \norm{\eta}_{L^\infty(\IOm)}\norm{\Delta_hz_{kh}}_{L^1(\IOm)},
    \end{aligned}
  \]
  where the stability of $R_h$ in $L^\infty(\Om)$ from~\cite{AHSchatz_1980a} for $N=2$ and
  from~\cite[Theorem 12]{LeykekhmanVexler:2016:II} for $N=3$  was used.  For the second term on the
  right-hand side of~\eqref{eq:1}, it follows
  \[
    \left\lvert \sum_{m=1}^{M} (\eta_m,[z_{kh}]_m)_\Om\right\rvert \le \sum_{m=1}^{M}
    \norm{\eta_m}_{L^\infty(\Om)}\norm{[z_{kh}]_m}_{L^1(\Om)}\le
    \norm{\eta}_{L^\infty(\IOm)}\sum_{m=1}^{M}\norm{[z_{kh}]_m}_{L^1(\Om)}.
  \]
  Finally, for the third term on the right-hand side of~\eqref{eq:1}, we obtain due
  to~\eqref{eq:boundDR} that
  \[
    \abs{(d_R(\cdot,\cdot,u^R) - d_R(\cdot,\cdot,\chi_{kh}),z_{kh})_{\IOm}}\le C_R
    \norm{\eta}_{L^\infty(\IOm)}\norm{z_{kh}}_{L^1(\IOm)}.
  \]
  Combining the previous estimates and applying Lemma~\ref{lemma:4} to the dual problem considered
  here as well as Lemma~\ref{lem:DeltaH} for $\norm{z_{kh}}_{L^1(\IOm)}$ leads to
  \[
    \begin{aligned}
      \xi_{kh,M}(x_0)
      &\le C_R\lh\norm{\eta}_{L^\infty(\IOm)}\left\{\norm{\Delta_hz_{kh}}_{L^1(\IOm)}
      +\sum_{m=1}^{M}\norm{[z_{kh}]_m}_{L^1(\Om)}+\norm{z_{kh}}_{L^1(\IOm)}\right\}\\
      &\le
      C_R\lh\left(\lk\right)^2\norm{\eta}_{L^\infty(\IOm)}\norm{\theta_M\tilde\delta_{x_0}}_{L^1(\IOm)}.
    \end{aligned}
  \]
  Using the bound
  \[
    \norm{\theta_M\tilde\delta_{x_0}}_{L^1(\IOm)}=\norm{\theta_M}_{L^1(I)}\norm{\tilde\delta_{x_0}}_{L^1(\Om)}\le
    C
  \]
  concludes the estimate of $\xi_{kh}$.
  Then, we get for the error
  \[
    \norm{e}_{L^\infty(\IOm)}\le
    \norm{\eta}_{L^\infty(\IOm)}+\norm{\xi_{kh}}_{L^\infty(\IOm)}\le
    C_R\lh\left(\lk\right)^2\norm{\eta}_{L^\infty(I;L^\infty(\Om))},
  \]
  which states the assertion.
\end{proof}

To formulate the boundedness result for $u_{kh}\in\Xkzh$, we require the following mild assumption on $k$ and $h$.

\begin{assumption}\label{ass:1}
  There exist $\sigma>0$ and a constant $C>0$ such that
  \[
    k\le Ch^\sigma.
  \]
\end{assumption}

\begin{theorem}\label{th:BoundUkh}
  Let the Assumptions~\ref{ass:fduz},~\ref{ass:2}, and~\ref{ass:1} be fulfilled.  Then, there exists
  $h_0>0$ and a constant $C > 0$ independent of $k$ and $h$ such that for all $h<h_0$ the solution
  $u_{kh}\in\Xkzh$ of~\eqref{eq:discHeat} fulfills
  \[
    \norm{u_{kh}}_{L^\infty(I\times\Omega)}\le \norm{u}_{L^\infty(\IOm)}+1.
  \]
\end{theorem}

\begin{proof}
  Let $R =  \norm{u}_{L^\infty(I\times \Omega)} +1$. By the boundedness of $u$, see
  Proposition~\ref{prop:ContSol}, we have $R<\infty$.  Due to this choice, it holds $u^R = u$.
  Using the estimate from Lemma~\ref{lem:estimate}, setting $\chi_{kh}=P_kP_hu$ and using the
  stability of the temporal $L^2$ projection $P_k$ in $L^\infty(\IOm)$, we get
  \[
    \begin{aligned}
      \norm{u &- u_{kh}^R}_{L^\infty(I\times \Omega)}\\
              &\le C_R\lh\left(\lk\right)^2
      \bigl\{\norm{u^R-P_ku^R}_{L^\infty(\IOm)}+\norm{P_k(u^R-P_hu^R)}_{L^\infty(\IOm)}\bigr\}\\
      &\le C_R\lh\left(\lk\right)^2
      \bigl\{\norm{u^R-P_ku^R}_{L^\infty(\IOm)}+\norm{u^R-P_hu^R}_{L^\infty(\IOm)}\bigr\}.
    \end{aligned}
  \]
  By standard estimates for $P_h$ and $P_k$ together with the regularity of $u$ from
  Theorem~\ref{th:regU}, it follows
  \[
    \begin{aligned}
      \norm{u - u_{kh}^R}_{L^\infty(I\times \Omega)}
      & \le
      C_R\lh\left(\lk\right)^2\bigl\{\norm{u-P_ku}_{L^\infty(\IOm)}+\norm{u-P_hu}_{L^\infty(\IOm)}\bigr\}\\
      &\le C_R \lh\left(\lk\right)^2 \bigl(k^\beta+h^\kappa\bigr) \norm{u}_{C^\beta(I;C^\kappa(\Om))}\\
      &\le C_R \lh\left(\lk\right)^2 \bigl(k^\beta+h^\kappa\bigr)
      \bigl\{\norm{f}_{L^p(I;L^r(\Om))}+\norm{u_0}_{U_{p_0,r_0}(\Om)}\bigr\}
    \end{aligned}
  \]
  Using Assumptions~\ref{ass:1}, it follows with $\delta=\min\{\sigma\beta,\kappa\}>0$
  \[
    \norm{u - u_{kh}^R}_{L^\infty(I\times \Omega)}\le C_R \lh^3 h^\delta.
  \]
  Consequently, there exists $h_0>0$, such that for all $h < h_0$ we have $\norm{u -
  u_{kh}^R}_{L^\infty(I\times \Omega)} \le 1$. This yields
  \[
    \norm{u_{kh}^R}_{L^\infty(I\times \Omega)} \le \norm{u}_{L^\infty(I\times \Omega)} +\norm{u -
    u_{kh}^R}_{L^\infty(I\times \Omega)}\le \norm{u}_{L^\infty(I\times \Omega)} +1 =R,
  \]
  and therefore $u_{kh} = u_{kh}^R$.  This gives the boundedness of $u_{kh}$. 
\end{proof}

\section{Error Estimates}\label{sec:estimates}

In this section, we provide (quasi) best approximation results and error estimates of the
discretization error between the continuous solution $u$ of~\eqref{eq:heatEquation} and the discrete
solution $u_{kh}$ of~\eqref{eq:discHeat} in various norms. Basis of all given estimates is the
boundedness of $u_{kh}$ given by Theorem~\ref{th:BoundUkh}.

We start with a best-approximation-type result in the $L^2(\IOm)$ norm.

\begin{theorem}\label{th:L2L2}
  Let the Assumption~\ref{ass:fduz},~\ref{ass:2} and~\ref{ass:1} be fulfilled. Further, let $u$ be
  the solution of~\eqref{eq:heatEquation}, and $u_{kh}\in\Xkzh$ be the solution
  of~\eqref{eq:discHeat} Then, it holds
  \[
    \norm{u-u_{kh}}_{L^2(\IOm)}\le C\bigl\{
    \norm{u-\chi_{kh}}_{L^2(\IOm)}+\norm{u-\Pi_ku}_{L^2(\IOm)}+\norm{u-R_hu}_{L^2(\IOm)}\bigr\}
  \]
  for any $\chi_{kh}\in\Xkzh$.
\end{theorem}
\begin{proof}
  Due to the boundedness of $u$ by Proposition~\ref{prop:ContSol} and the boundedness of $u_{kh}$ by
  Theorem~\ref{th:BoundUkh}, we have
  \[
    R_u=\norm{u}_{L^\infty(\IOm)}<\infty \quad\text{and}\quad
    R_{u_{kh}}=\sup_{k,h}\norm{u_{kh}}_{L^\infty(\IOm)}<\infty.
  \]
  Choosing $R=\max(R_u,R_{u_{kh}})$ in Lemma~\ref{lem:estimate}, we directly obtain $u=u^R$ and
  $u_{kh}=u_{kh}^R$. Proceeding as in the proof of Lemma~\ref{lem:estimate}, we decompose
  \[
    e=u-u_{kh}=(u-\chi_{kh})+(\chi_{kh}-u_{kh})=\eta+\xi_{kh}
  \]
  and introduce the following dual problem for $z_{kh}\in\Xkzh$:
  \[
    B(\varphi_{kh},z_{kh}) + (b \varphi_{kh},z_{kh})_{\IOm} =
    (\xi_{kh},\varphi_{kh})_{\IOm},\quad\forall\varphi_{kh}\in\Xkzh
  \]
  with $b$ as in the proof of Lemma~\ref{lem:estimate}. Testing with $\phi_{kh}=\xi_{kh}$ yields
  \begin{equation}\label{eq:7}
    \begin{aligned}
      \norm{\xi_{kh}}_{L^2(\IOm)}^2&=
      B(\xi_{kh},z_{kh})+(b\xi_{kh},z_{kh})_{\IOm}\\
      &=-(\nabla\eta,\nabla z_{kh})_{\IOm}+\sum_{m=1}^{M} (\eta_m,[z_{kh}]_m)_\Om-
      (d_R(\cdot,\cdot,u) - d_R(\cdot,\cdot,\chi_{kh}),z_{kh})_{\IOm}.\\
    \end{aligned}
  \end{equation}
  For the first term on the right-hand side of~\eqref{eq:7}, we get
  \[
    \abs{(\nabla\eta,\nabla z_{kh})_{\IOm}}\ = \abs{(R_h\eta,\Delta_h z_{kh})_{\IOm}}\\
    \le\norm{R_h\eta}_{L^2(\IOm)}\norm{\Delta_hz_{kh}}_{L^2(\IOm)}.
  \]
  For the second term on the right-hand side of~\eqref{eq:7}, it follows from the definition of
  $\Pi_k$ that
  \[
    \eta_m = u(t_m)-\chi_{kh,m} = u(t_m)-\chi_{kh}(t_m) = (\Pi_k u)(t_m) - \Pi_k (\chi_{kh})(t_m) =
    (\Pi_k\eta)_m
  \]
  and thus
  \[
    \begin{aligned}
      \left\lvert \sum_{m=1}^{M} (\eta_m,[z_{kh}]_m)_\Om\right\rvert
      &=\left\lvert \sum_{m=1}^{M} ((\Pi_k\eta)_m,[z_{kh}]_m)_\Om\right\rvert
      \le \sum_{m=1}^{M}\norm{(\Pi_k\eta)_m}_{L^2(\Om)}\norm{[z_{kh}]_m}_{L^2(\Om)}\\
      &\le\left(\sum_{m=1}^{M}k_m\norm{\Pi_k\eta}_{L^2(\Om)}^2\right)^{\frac12}
      \left(\sum_{m=1}^{M}k_m^{-1}\norm{[z_{kh}]_m}_{L^2(\Om)}^2\right)^{\frac12}\\
      &=\norm{\Pi_k\eta}_{L^2(\IOm)}\left(\sum_{m=1}^{M}k_m^{-1}\norm{[z_{kh}]_m}_{L^2(\Om)}^2\right)^{\frac12}.
    \end{aligned}
  \]
  Finally, for the third term on the right-hand side of~\eqref{eq:7}, we obtain due
  to~\eqref{eq:boundDR}
  \[
    \abs{(d_R(\cdot,\cdot,u) - d_R(\cdot,\cdot,\chi_{kh}),z_{kh})_{\IOm}}\le C_R
    \norm{\eta}_{L^2(\IOm)}\norm{z_{kh}}_{L^2(\IOm)}.
  \]
  It remains to bound the arising terms involving $z_{kh}$. By Lemma~\ref{lemma:1} applied to the
  dual problem for $z_{kh}$, we have $\norm{z_{kh}}_{L^\infty(I;L^2(\Om))}\le
  \norm{\xi_{kh}}_{L^1(I;L^2(\Om))}$ and consequently
  \[
    \norm{bz_{kh}}_{L^2(\IOm)}\le \norm{b}_{L^\infty(\IOm)}\norm{z_{kh}}_{L^2(\IOm)}\le
    \norm{b}_{L^\infty(\IOm)}\norm{\xi_{kh}}_{L^2(\IOm)}.
  \]
  Then,~\cite[Corollary~4.2]{MeidnerVexler:2008:I} applied to the rewritten dual problem for $z_{kh}$
  \[
    B(\varphi_{kh},z_{kh}) =  (\xi_{kh}-bz_{kh},\varphi_{kh})_{\IOm},\quad\forall\varphi_{kh}\in\Xkzh
  \]
  yields
  \[
    \begin{aligned}
      \norm{\Delta_hz_{kh}}_{L^2(\IOm)}+
      \left(\sum_{m=1}^{M}k_m^{-1}\norm{[z_{kh}]_m}_{L^2(\Om)}^2\right)^{\frac12}
      &\le\norm{\xi_{kh}-bz_{kh}}_{L^2(\IOm)}\\
      &\le \bigl\{1+\norm{b}_{L^\infty(\IOm)}\bigr\}\norm{\xi_{kh}}_{L^2(\IOm)}.
    \end{aligned}
  \]
  Using Lemma~\ref{lem:DeltaH} to bound $\norm{z_{kh}}_{L^2(\IOm)}$ by
  $\norm{\Delta_hz_{kh}}_{L^2(\IOm)}$ and the boundedness of $\norm{b}_{L^\infty(\IOm)}$ due
  to~\eqref{eq:boundDR}, we obtain
  \[
    \norm{\xi_{kh}}_{L^2(\IOm)}\le C
    \bigl\{\norm{\eta}_{L^2(\IOm)}+\norm{\Pi_k\eta}_{L^2(\IOm)}+\norm{R_h\eta}_{L^2(\IOm)}\bigr\}.
  \]
  Then, the triangle inequality implies the assertion.
\end{proof}

Under slightly strengthened assumptions on $f$ and $u_0$ Theorem~\ref{th:L2L2} yields an error
estimate in the $L^2(\IOm)$ norm of optimal order.

\begin{corollary}\label{cor:L2L2}
  Let the Assumption~\ref{ass:fduz},~\ref{ass:2} and~\ref{ass:1} be fulfilled and additionally
  $p,r\ge 2$ and $u_0\in H^1_0(\Om)$. Then, for the solution $u$ of~\eqref{eq:heatEquation}, it
  holds $u\in H^1(I;L^2(\Om))\cap L^2(I;H^2(\Om))$ with
  \[
    \norm{\pa_t u}_{L^2(\IOm)}+\norm{\nabla^2 u}_{L^2(\IOm)}\le
    C\bigl\{\norm{f}_{L^p(I;L^r(\Om))}+\norm{\nabla u_0}_{L^2(\Om)}+\norm{u_0}_{L^\infty(\Om)}\bigr\}.
  \]
  Further, for the error between $u$ and the solution $u_{kh}\in\Xkzh$ of~\eqref{eq:discHeat}, it
  holds
  \[
    \norm{u-u_{kh}}_{L^2(\IOm)}\le
    C(k+h^2)\bigl\{\norm{f}_{L^p(I;L^r(\Om))}+\norm{\nabla u_0}_{L^2(\Om)}+\norm{u_0}_{L^\infty(\Om)}\bigr\}.
  \]
\end{corollary}
\begin{proof}
  By putting the nonlinearity $d$ to the right-hand side as
  \[
    \pa_t u-\Delta u =f- d(\cdot,\cdot,u),
  \]
  regularity theory for the linear equation (cf., e.g.,~\cite[Chapter~7, Theorem~5]{Evans:2010}) yields as in the proof
  of Theorem~\ref{th:regU} by means of Proposition~\ref{prop:ContSol} that
  \[
    \begin{aligned}
      \norm{\pa_t u}_{L^2(\IOm)}+\norm{\nabla^2 u}_{L^2(\IOm)}
      &\le C\bigl\{\norm{f-d(\cdot,\cdot,u)}_{L^2(\IOm)} + \norm{\nabla u_0}_{L^2(\Om)}\bigr\}\\
      &\le C\bigl\{\norm{f}_{L^2(\IOm)}+\norm{u}_{L^\infty(\IOm)} + \norm{\nabla u_0}_{L^2(\Om)}\bigr\}\\
      &\le
      C\bigl\{\norm{f}_{L^p(I;L^r(\Om))}+\norm{\nabla u_0}_{L^2(\Om)}+\norm{u_0}_{L^\infty(\Om)}\bigr\},
    \end{aligned}
  \]
  since $p,r\ge 2$.

  From Theorem~\ref{th:L2L2}, we have
  \[
    \norm{u-u_{kh}}_{L^2(\IOm)}\le C\bigl\{
    \norm{u-\chi_{kh}}_{L^2(\IOm)}+\norm{u-\Pi_ku}_{L^2(\IOm)}+\norm{u-R_hu}_{L^2(\IOm)}\bigr\}.
  \]
  Choosing $\chi_{kh}=P_kP_h u$ as in the proof of Theorem~\ref{th:BoundUkh}, we get by the
  stability of $P_k$ in $L^2(\IOm)$
  \[
    \norm{u-\chi_{kh}}_{L^2(\IOm)}\le
    C\bigl\{\norm{u-P_ku}_{L^2(\IOm)}+\norm{u-P_hu}_{L^2(\IOm)}\bigr\}.
  \]
  Then, the standard estimates
  \[
    \begin{aligned}
      \norm{u-P_ku}_{L^2(\IOm)}+\norm{u-\Pi_ku}_{L^2(\IOm)}&\le Ck\norm{\pa_tu}_{L^2(\IOm)},\\
      \norm{u-P_hu}_{L^2(\IOm)}+\norm{u-R_hu}_{L^2(\IOm)}&\le Ch^2\norm{\nabla^2u}_{L^2(\IOm)}
    \end{aligned}
  \]
  yield the assertion.
\end{proof}

Next, we derive a best-approximation-type result in the $L^\infty(I;L^2(\Om))$ norm.

\begin{theorem}\label{th:L8L2}
  Let the Assumption~\ref{ass:fduz},~\ref{ass:2} and~\ref{ass:1} be fulfilled. Further, let $u$ be
  the solution of~\eqref{eq:heatEquation}, and $u_{kh}\in\Xkzh$ be the solution
  of~\eqref{eq:discHeat} Then, it holds for all $1\le \hat p\le\infty$
  \[
    \norm{u-u_{kh}}_{L^\infty(I;L^2(\Om))}\le
    C\lk\bigl\{\norm{u-\chi_{kh}}_{L^\infty(I;L^2(\Om))}+k^{-\frac1{\hat p}}\norm{u-R_hu}_{L^{\hat
    p}(I;L^2(\Om))}\bigr\}
  \]
\end{theorem}
\begin{proof}
  Again, due to the boundedness of $u$ by Proposition~\ref{prop:ContSol} and the boundedness of $u_{kh}$ by
  Theorem~\ref{th:BoundUkh}, we have
  \[
    R_u=\norm{u}_{L^\infty(\IOm)}<\infty \quad\text{and}\quad
    R_{u_{kh}}=\sup_{k,h}\norm{u_{kh}}_{L^\infty(\IOm)}<\infty.
  \]
  Choosing $R=\max(R_u,R_{u_{kh}})$ in Lemma~\ref{lem:estimate}, we directly obtain $u^R=u$ and
  $u_{kh}^R=u_{kh}$. Proceeding as in the proof of Lemma~\ref{lem:estimate}, we decompose
  \[
    e=u-u_{kh}=(u-\chi_{kh})+(\chi_{kh}-u_{kh})=\eta+\xi_{kh}
  \]
  and introduce the following dual problem for $z_{kh}\in\Xkzh$:
  \[
    B(\varphi_{kh},z_{kh}) + (b \varphi_{kh},z_{kh})_{\IOm} =
    (\xi_{kh,M}\theta_M,\varphi_{kh})_{\IOm},\quad\forall\varphi_{kh}\in\Xkzh.
  \]
  with $b$ and $\theta_M$ as in the proof of Lemma~\ref{lem:estimate}. Testing with
  $\phi_{kh}=\xi_{kh}$ yields
  \begin{equation}\label{eq:6}
    \begin{aligned}
      \norm{\xi_{kh,M}}_{L^2(\Om)}^2&=
      B(\xi_{kh},z_{kh})+(b\xi_{kh},z_{kh})_{\IOm}\\
      &=-(\nabla\eta,\nabla z_{kh})_{\IOm}+\sum_{m=1}^{M} (\eta_m,[z_{kh}]_m)_\Om-
      (d_R(\cdot,\cdot,u) - d_R(\cdot,\cdot,\chi_{kh}),z_{kh})_{\IOm}.
    \end{aligned}
  \end{equation}
  For the first term on the right-hand side of~\eqref{eq:6}, we get by an inverse estimate for
  $\frac1{\hat p}+\frac1{\hat p'}=1$ that
  \[
    \begin{aligned}
      \abs{(\nabla\eta,\nabla z_{kh})_{\IOm}}\
      &= \abs{(R_h\eta,\Delta_h z_{kh})_{\IOm}}\le \abs{(u- R_h u,\Delta_h z_{kh})_{\IOm}} + \abs{(\eta,\Delta_h z_{kh})_{\IOm}}\\
      &\le\norm{u-R_hu}_{L^{\hat p}(I;L^2(\Om))}\norm{\Delta_hz_{kh}}_{L^{\hat p'}(I;L^2(\Om))}
      +\norm{\eta}_{L^\infty(I;L^2(\Om))}\norm{\Delta_hz_{kh}}_{L^1(I;L^2(\Om))}\\
      &\le C \bigl\{k^{-\frac1{\hat p}}\norm{u-R_hu}_{L^{\hat p}(I;L^2(\Om))}
    +\norm{\eta}_{L^\infty(I;L^2(\Om))}\bigr\}\norm{\Delta_hz_{kh}}_{L^1(I;L^2(\Om))}.
    \end{aligned}
  \]
  For the second term on the right-hand side of~\eqref{eq:6}, we obtain
  \[
    \left\lvert \sum_{m=1}^{M} (\eta_m,[z_{kh}]_m)_\Om\right\rvert \le \sum_{m=1}^{M}
    \norm{\eta_m}_{L^2(\Om)}\norm{[z_{kh}]_m}_{L^2(\Om)}\le
    \norm{\eta}_{L^\infty(I;L^2(\Om))}\sum_{m=1}^{M}\norm{[z_{kh}]_m}_{L^2(\Om)}.
  \]
  Finally, for the third term on the right-hand side of~\eqref{eq:6}, we obtain due
  to~\eqref{eq:boundDR} that
  \[
    \abs{(d_R(\cdot,\cdot,u) - d_R(\cdot,\cdot,\chi_{kh}),z_{kh})_{\IOm}}\le C_R
    \norm{\eta}_{L^\infty(I;L^2(\Om))}\norm{z_{kh}}_{L^1(I;L^2(\Om))}.
  \]
  It remains to bound the arising terms involving $z_{kh}$. By Lemma~\ref{lemma:1} applied to the
  dual problem for $z_{kh}$, we have $\norm{z_{kh}}_{L^\infty(I;L^2(\Om))}\le
  \norm{\xi_{kh,M}\theta_M}_{L^1(I;L^2(\Om))}$ and consequently
  \[
    \begin{aligned}
      \norm{bz_{kh}}_{L^1(I;L^2(\Om))}
      &\le \norm{b}_{L^\infty(\IOm)}\norm{z_{kh}}_{L^1(I;L^2(\Om))}\\
      &\le \norm{b}_{L^\infty(\IOm)}\norm{\xi_{kh,M}\theta_M}_{L^1(I;L^2(\Om)}
      =\norm{b}_{L^\infty(\IOm)}\norm{\xi_{kh,M}}_{L^2(\Om)}
    \end{aligned}
  \]
  due to the properties of $\theta_M$. By~\cite[Theorem~11]{LeykekhmanVexler:2017:II} applied to the
  rewritten dual problem for $z_{kh}$
  \[
    B(\varphi_{kh},z_{kh}) =  (\xi_{kh,M}\theta_M-bz_{kh},\varphi_{kh})_{\IOm},\quad\forall\varphi_{kh}\in\Xkzh
  \]
  yields
  \[
    \begin{aligned}
      \norm{\Delta_hz_{kh}}_{L^1(I;L^2(\Om))}+ \sum_{m=1}^{M}\norm{[z_{kh}]_m}_{L^2(\Om)}
      &\le C\lk \norm{\xi_{kh,M}\theta_M-bz_{kh}}_{L^1(I;L^2(\Om))}\\
      &\le C\lk \bigl\{1+\norm{b}_{L^\infty(\IOm)}\bigr\}\norm{\xi_{kh,M}}_{L^2(\Om)}.
    \end{aligned}
  \]
  Using Lemma~\ref{lem:DeltaH} for $\norm{z_{kh}}_{L^1(I;L^2(\Om))}$ and the boundedness of
  $\norm{b}_{L^\infty(\IOm)}$ due to~\eqref{ass:b}, we obtain
  \[
    \norm{\xi_{kh,M}}_{L^2(\Om)}\le C\lk
    \bigl\{\norm{\eta}_{L^\infty(I;L^2(\Om))}+Ck^{-\frac1{\hat p}}\norm{u-R_hu}_{L^{\hat
    p}(I;L^2(\Om))}\bigr\},
  \]
  which yields the assertion.
\end{proof}

Under further strengthened assumptions on $f$ and $u_d$, also this quasi best approximation result
implies an error estimate of optimal (up to logarithmic terms) order.

\begin{corollary}\label{cor:L8L2}
  Let the Assumption~\ref{ass:fduz},~\ref{ass:2} and~\ref{ass:1} be fulfilled and additionally $r\ge
  2$, $f\in L^\infty(I,L^r(\Om))$, and $u_0\in H^2(\Om)\cap H^1_0(\Om)$. Then, for the solution $u$
  of~\eqref{eq:heatEquation}, it holds $u\in W^{1,\hat p}(I;L^2(\Om))\cap L^{\hat p}(I;H^2(\Om))$
  for all $1< \hat p<\infty$ and there exists a constant $C_{\hat p}\le C\frac{\hat p^2}{\hat p-1}$
  with
  \[
    \norm{\pa_t u}_{L^{\hat p}(I;L^2(\Om))}+\norm{\nabla^2 u}_{L^{\hat p}(I;L^2(\Om))}\le
    C_{\hat p}\bigl\{\norm{f}_{L^\infty(I;L^r(\Om))}+\norm{\nabla^2 u_0}_{L^2(\Om)}\bigr\}.
  \]
  Further, for the error between $u$ and the solution $u_{kh}\in\Xkzh$ of~\eqref{eq:discHeat}, it
  holds
  \[
    \norm{u-u_{kh}}_{L^\infty(I;L^2(\Om))}\le
    C(k+h^2)\left(\lk\right)^2\bigl\{\norm{f}_{L^\infty(I;L^r(\Om))}+\norm{\nabla^2 u_0}_{L^2(\Om)}\bigr\}.
  \]
\end{corollary}

\begin{proof}
  We put the nonlinearity $d$ to the right-hand side as
  \[
    \begin{aligned}
      \pa_t u-\Delta u &=f- d(\cdot,\cdot,u)&\quad&\text{in }I\times\Omega,\\
      u(0)&=u_0&&\text{on }\Omega,
    \end{aligned}
  \]
  and split the solution as $u=v+w$ where $v$ solves~\eqref{eq:linRhs} with $g=f-d(\cdot,\cdot,u)$
  and $w$ solves~\eqref{eq:linU0}. Then the Propositions~\ref{prop:ContHoelderRhs}
  and~\ref{prop:ContHoelderU0} imply
  \[
    \norm{\pa_t v}_{L^{\hat p}(I;L^2(\Om))}+\norm{\nabla^2 v}_{L^{\hat p}(I;L^2(\Om))}
    \le C_{\hat p}\norm{f-d(\cdot,\cdot,u)}_{L^{\hat p}(I;L^2(\Om))}.
  \]
  with $C_{\hat p}\le C\frac{\hat p^2}{\hat p-1}$ and
  \[
    \norm{\pa_t w}_{L^\infty(I;L^2(\Om))}+\norm{\nabla^2 w}_{L^\infty(I;L^2(\Om))}
    \le C\norm{\nabla^2 u_0}_{L^2(\Om)}.
  \]
  Combining these estimates and proceeding similarly to the proof of Theorem~\ref{th:regU} by means
  of Proposition~\ref{prop:ContSol} then implies
  \[
    \begin{aligned}
      \norm{\pa_t u}_{L^{\hat p}(I;L^2(\Om))}+\norm{\nabla^2 u}_{L^{\hat p}(I;L^2(\Om))}
      &\le C_{\hat p}\norm{f-d(\cdot,\cdot,u)}_{L^{\hat p}(I;L^2(\Om))} + C\norm{\nabla^2 u_0}_{L^2(\Om)}\\
      &\le C_{\hat p}\bigl\{\norm{f}_{L^\infty(I;L^2(\Om))}+\norm{u}_{L^\infty(\IOm)}\bigr\} + C\norm{\nabla^2 u_0}_{L^2(\Om)}\\
      &\le C_{\hat p}\bigl\{\norm{f}_{L^\infty(I;L^r(\Om))}+\norm{\nabla^2u_0}_{L^2(\Om)}\bigr\},
    \end{aligned}
  \]
  since $r\ge 2$ and $\hat p<\infty$.

  From Theorem~\ref{th:L8L2}, we have
  \[
    \norm{u-u_{kh}}_{L^\infty(I;L^2(\Om))}\le
    C\lk\bigl\{\norm{u-\chi_{kh}}_{L^\infty(I;L^2(\Om))}+k^{-\frac 1{\hat p}}\norm{u-R_hu}_{L^{\hat
    p}(I;L^2(\Om))}\bigr\}.
  \]
  Choosing $\chi_{kh}=P_kP_h u$ as in the proof of Theorem~\ref{th:BoundUkh}, we get
  \[
    \begin{aligned}
      \norm{u-\chi_{kh}}_{L^\infty(I;L^2(\Om))}
      &\le\norm{u-P_ku}_{L^\infty(I;L^2(\Om))}+\norm{P_k(u-P_hu)}_{L^\infty(I;L^2(\Om))}\\
      &\le\norm{u-P_ku}_{L^\infty(I;L^2(\Om))}+C k^{-\frac1{\hat p}}\norm{P_k(u-P_hu)}_{L^{\hat p}(I;L^2(\Om))}\\
      &\le\norm{u-P_ku}_{L^\infty(I;L^2(\Om))}+C k^{-\frac1{\hat p}}\norm{u-P_hu}_{L^{\hat p}(I;L^2(\Om))}.
    \end{aligned}
  \]
  From the stability of $P_k$ in $L^\infty(I;L^2(\Om)$ and standard interpolation estimates, we have
  \[
    \norm{u-P_ku}_{L^\infty(I;L^2(\Om))}\le Ck^{1-\frac1{\hat p}}\norm{\pa_tu}_{L^{\hat p}(I;L^2(\Om))}.
  \]
  Further, standard estimates for $\norm{u-P_hu}_{L^2(\Om)}$ and $\norm{u-R_hu}_{L^2(\Om)}$ imply
  \[
    \norm{u-P_hu}_{L^{\hat p}(I;L^2(\Om))}+\norm{u-R_hu}_{L^{\hat p}(I;L^2(\Om))}\le Ch^2\norm{\nabla^2u}_{L^{\hat p}(I;L^2(\Om))}.
  \]
  Using these estimates, we get
  \[
  \begin{aligned}
    \norm{u-u_{kh}}_{L^\infty(I;L^2(\Om))}&\le C \lk k^{-\frac1{\hat p}}\bigl\{ k
  \norm{\pa_tu}_{L^{\hat p}(I;L^2(\Om))}+h^2\norm{\nabla^2u}_{L^{\hat p}(I;L^2(\Om))}\bigr\}\\ &\le
    C_{\hat p} k^{-\frac1{\hat p}}  (k+h^2)\lk\bigl\{\norm{f}_{L^\infty(I;L^r(\Om))}+\norm{\nabla^2
    u_0}_{L^2(\Om)}\bigr\}.
  \end{aligned}
  \]
  Then, by setting $\hat p=\lk$ we have $C_{\hat p}k^{-\frac1{\hat p}}\le C\lk$, since $\frac Tk\ge
  4$ by assumption. This implies the assertion.
\end{proof}

Finally, in the following Theorem, a best approximation result in $L^\infty(\IOm)$ is stated. This
is a direct consequence of Theorem~\ref{th:BoundUkh}.

\begin{theorem}\label{th:L8L8}
  Let the Assumption~\ref{ass:fduz},~\ref{ass:2} and~\ref{ass:1} be fulfilled. Further, let $u$ be
  the solution of~\eqref{eq:heatEquation}, and $u_{kh}\in\Xkzh$ be the solution
  of~\eqref{eq:discHeat} Then, it holds
  \[
    \norm{u-u_{kh}}_{L^\infty(\IOm)}\le C\lh\left(\lk\right)^2\norm{u-\chi_{kh}}_{L^\infty(\IOm)}
  \]
  for any $\chi_{kh}\in\Xkzh$.
\end{theorem}
\begin{proof}
  Due to the boundedness of $u$ by Proposition~\ref{prop:ContSol} and the boundedness of $u_{kh}$ by
  Theorem~\ref{th:BoundUkh}, we have
  \[
    R_u=\norm{u}_{L^\infty(\IOm)}<\infty \quad\text{and}\quad
    R_{u_{kh}}=\sup_{k,h}\norm{u_{kh}}_{L^\infty(\IOm)}<\infty.
  \]
  Choosing $R=\max(R_u,R_{u_{kh}})$ in Lemma~\ref{lem:estimate}, we directly obtain
  \[
    \norm{u-u_{kh}}_{L^\infty(\IOm)}=\norm{u^R-u^R_{kh}}_{L^\infty(\IOm)}\le
    C\lh\left(\lk\right)^2\norm{u-\chi_{kh}}_{L^\infty(\IOm)}.
  \]
  This concludes the short proof.
  \end{proof}

\bibliography{lit}
\bibliographystyle{abbrv}

\end{document}